\documentclass[twoside,11pt]{article}

\usepackage{graphicx,amsmath,latexsym,amssymb,amsthm}
\usepackage[colorlinks,urlcolor=blue]{hyperref}
\usepackage{pslatex}

\date{}
\newtheorem{thm}{Theorem}[section]
\newtheorem{lem}[thm]{Lemma}

\newtheorem{rem}[thm]{Remark}

\newtheorem{de}[thm]{Definition}

\numberwithin{equation}{section}

\newcommand{\Z}{{\mathbb{Z}}}
\newcommand{\R}{{\mathbb{R}}}
\newcommand{\C}{{\mathbb{C}}}
\newcommand{\T}{{\mathbb{T}}}

\begin{document}
\setlength{\unitlength}{1cm}
\vskip1.5cm
% ---------------------TITLE----------------------------------

\centerline {\bf Spectral Envelopes - A Preliminary Report}
% Put your title here
% ---------------------NAME---------------------------------

\vskip.8cm \centerline {\textbf{Wayne Lawton$^1$}\footnote{ Corresponding author e-mail : {\tt scwlw@mahidol.ac.th}\\
\\ {Copyright \copyright\, 2012  by the AMM2012. All rights reserve.}}}

\vskip.5cm
% Address 1 -------------------------------------------------------------------

\centerline {$^1$Department of Mathematics, Faculty of Science}
\centerline {Mahidol University, Rama 6 Road, Bangkok 10400, THAILAND}
%\centerline {$^1$Name of Department, Name of Faculty}
%\centerline {Name of University, University Address, COUNTRY}
%\centerline {{\tt first\_author@address}}

% ABSTRACT----------------------------------------------------------------

\vskip.5cm \hskip-.5cm{\small{\bf Abstract :}
The spectral envelope $S(F)$ of a subset $F$
of integers is the set of probability measures on the circle group $\T = \R/\Z$ that are weak$^{*}$ limits of squared moduli of norm one trigonometric polynomials having frequencies in $F.$ Clearly, if $\mu \in S(F)$ then the support of its Fourier transform $\widehat \mu$ is contained in the difference set $F - F.$ However, the converse generally fails and an explicit characterization of the elements in $S(F)$ is extremely elusive. We associate to $F$ a symbolic dynamical system $(\Omega(F),\sigma),$ where $\sigma$ is the shift homeomorphism on the product space $\{0,1\}^{\Z}$ and $\Omega(F)$ is the closure of the orbit under $\sigma$ of the characteristic function $\chi_{F}$ of $F.$ Analytic properties of $S(F)$ are related to dynamical properties of the binary sequence $\chi_F.$ The Riemann-Lebesque lemma implies that if $\chi_F$ is recurrent, or more specifically minimal, then $S(F)$ is convex and hence, by the Krein-Milman theorem, $S(F)$ equals the closure of the convex hull of its set of extreme points $S_e(F).$ In this paper we
(i) review the relationship between these concepts and the special case of the still open 1959 Kadison-Singer problem called Feichtinger's conjecture for exponential functions,
(ii) derive partial characterizations of elements in $S_e(F),$ for minimal $\chi_F,$ in terms of ergodic properties of $(\Omega(F),\lambda,\sigma),$ where $\lambda$ is a $\sigma-$ invariant probability measure whose existence is ensured by the 1937 Krylov-Bogolyubov theorem,
(iii) refine previous numerical studies of the Morse-Thue minimal binary sequence by exploiting a new MATLAB algorithm for computing smallest eigenvalues of $4,000,000 \times 4,000,000$ matrices,
(iv) describe recent results characterizing $S(F)$ for certain Bohr sets $F$ related to quasicrystals,
(v) extend these concepts to general discrete groups including those with Kazhdan's T-property, such as $SL(n,Z), n \geq 3,$ which can be characterized by several equivalent properties
such as: any sequence of positive definite functions converging to 1 uniformly on compact subsets converges uniformly. This exotic property may be useful to construct a counterexample to the generalization of Feichtinger's conjecture and hence provide a no answer to the question of Kadison and Singer which they themselves tended to suspect.



%%%%%%%  MAIN BODY  %%%%%%%%%%%%%%%%%%%%%
\section{Introduction}
We let $\Z, \, \R$ and $\C$ denote the integer, real, and complex numbers and $\T = \R/\Z$ denote the circle group.
For any subset $F \subseteq \Z$ we define the set of trigonometric polynomials
$$
	P(F) = \{\, p \, : \, \T  \rightarrow \C \, : \, p(x) = \sum_{k \in F} \, c_k \, e^{2\pi i kx}, \ x \in \T, \ c_k \in \C , \
        \int_{x \in \T} |p(x)|^2 \, dx = 1 \ \}
$$
with norm $1$ and whose frequencies are in $F,$ and we define the set of probability measures
$$
	M(F) = \{\, \hbox{probability measures } \mu  \hbox{ on } \T \, :
\, k \notin F \implies \widehat \mu(k) = \int_{x \in \T} e^{-2\pi i k x} \, d\mu(x) = 0 \  \}
$$
whose frequencies are in $F.$ We identify absolutely continuous measures on $\T$ with their Radon-Nikodym derivatives and define
$
    s\, : \, P(F) \rightarrow M(F-F)
$
by
$
    s(p) = |p|^2, \ \ p \in P(F).
$
\begin{de}
The {\bf spectral envelope} $S(F)$ is the weak$^{\, *}$ closure of $s(P(F)).$
\end{de}
\begin{lem}\label{lem1}
For $F \subseteq \Z,$ $m \in \Z,$ and $\tau \in \T$ the following equalities hold:
$S(F+m) = S(F),$ $S(F) + \tau = S(F),$ $M(F) + \tau = M(F),$ and $S(F) \subseteq M(F-F).$
$M(F)$ is an ideal in $M(\Z).$
If the cardinality $|F|$ of $F$ is finite
then $s(P(F)) = S(F),$ $\dim \, M(F-F) = |F-F|,$ and $\dim \, S(F) \leq 2|F|-1.$
\end{lem}
\begin{proof}
The first assertions are obvious, the last is implied by
Sard's theorem \cite{S42}, \cite{SS42}.
\end{proof}
\begin{lem}\label{lem2}
If $F$ is the set of points in an arithmetic sequence then $S(F) = M(F-F).$
\end{lem}
\begin{proof}
If $J \in \Z, d \geq 0,$ and $F = \{J,J+1,...,J+d\}$ then
$F-F = \{ \, -d,...,d \, \}$
and the Fej\'{e}r-Riesz spectral factorization lemma (\cite{RN55}, p. 117),
conjectured by Fej\'{e}r \cite{F15} and proved by Riesz \cite{R15}, implies
that $S(F) = M(F-F).$ Furthermore $S(F)$ contains the Fej\'{e}r kernel
$$
    K_d(x) = \left| \, \frac{1}{\sqrt {d+1} } \sum_{j=0}^{d} e^{\, 2\pi i j x} \, \right|^{\, 2}, \ \ x \in \T
$$
which converges weak$^{\, *}$ to Dirac's measure $\delta \in M(\Z)$ as $d \rightarrow \infty.$
These results easily extend to arithmetic sequences whose gap between consecutive points is larger than $1.$
\end{proof}
\noindent In general if $\mu \in M(\{-d,...,d\})$ there are $2^d$ distinct polynomials $p \in P(\{0,...,d\})$
that satisfy $p(0) > 0$ and $|p|^2 = \mu.$ The problem of computing $p$ from $|p|$ arises in imaging applications,
such as X-ray crystallography, speckle interferometry and holography, where it is called the phase retrieval problem
\cite{FW86}. If $p$ and $q$ are multidimensional trigonometric polynomials or entire functions of exponential type
and $|p| = |q,$ then generally either $q = cp$ or $q = c\overline p$ where $c \in \C$ and $|c| = 1,$ \cite{LM87}.
We propose the following inverse problem:
\begin{rem}\label{rem1}
For $F \subset \Z$ and $\mu \in s(P(F))$ efficiently compute $p \in P(F)$ such that $|p|^2 = \mu.$
\end{rem}
\noindent For an arbitrary subset $F \subseteq \Z$ we define the following analogue of the Fej\'{e}r kernels
$$
    K_d(F)(x) = \left| \, \frac{1}{\sqrt {|F\cap \{-d,...,d\}|} } \sum_{j \in F\cap \{-d,...,d\}} e^{\, 2\pi i j x} \, \right|^{\, 2}, \ \ x \in \T, \ d \geq 0.
$$
If $A \subseteq \T$ is Lebesque measurable and has measure $\lambda(A) \in (0,1]$ and $\alpha > 0$ is irrational, then Birkoff's ergodic theorem
\cite{B31} implies that for almost all $\beta \in \T$ the set
$$
    F = \{ \, n \in \Z \, : \, n \, \alpha + \beta \in A \, \}
$$
has density $\lambda(A).$ In this case we have the following weak$^{\, *}$ limit
$$
    \lim_{d \rightarrow \infty} K_d(F)(x) = \frac{1}{\lambda(A)} \, \sum_{j \in \Z} |\widehat \chi_A(j)|^{\, 2} \, \delta(j \, \alpha + x).
$$
If $A$ is a nonempty open set then $F$ is a special case of sets named after Harald Bohr, who
initiated the theory of uniformly almost periodic functions \cite{BO52}, and whose number theoretic
properties were studied by Ruzsa (\cite{GR09}, Definition 2.5.1). Then $\chi_F$ is a minimal binary sequence.
Another example of a minimal binary sequences was constructed by Thue \cite{TH06} and Morse \cite{MO21}. If $F$ is the support of the
Morse-Thue minimal sequence then $K_d(F)$ converges to a measure with no discrete (delta) function components but whose support has Lebesque
measure $0$ \cite{BG08}. This sequence describes the structure of quasicrystals that have been fabricated in laboratories and the measure describes the observed X-ray diffraction spectra of these quasicrystals \cite{N04}. 
\section{Kadison-Singer Problem}
\noindent In 1959 Kadison and Singer formulated a yet unsolved problem \cite{KS59} that emerged from Dirac's formulation of
quantum mechanics \cite{D30}. Casazza and Tremain \cite{CA06} proved that a yes answer to their problem is equivalent to the
Feichtinger conjecture. A special case of the Feichtinger conjecture is called the Feichtinger conjecture for exponentials. 
The deepest results in this area are arguably those of Bourgain and Tzafriri \cite{BT87}, \cite{BT91} and the recent results of
Spielman and Srivastava \cite{SS12}.
Independently and using different methods, Paulson \cite{PA11} and I \cite{LA10} derived a relationship between the Feichtinger 
conjecture for exponentials with the concepts of syndetic sets and minimal sequences that are central to topological dynamics \cite{G46}, \cite{GH55}. 
In \cite{LA10} I derived results related to sampling as studied by Olevskii and Ulanovskii in \cite{O08}.
In \cite{LA11} I formulated the concept of spectral envelopes and derived its relationships to Feichtingers conjecture for exponentials,
presented preliminary numerical results, and outlined approaches to relate spectral envelopes to the spectral theory based on the theory developed by
Queffelec in \cite{Q87}.
\section{Characterizing Spectral Envelopes}
\noindent In \cite{LA12} I derived a result that easily implies that $S(F) = M(F-F)$ for a certain family of Bohr sets $F.$ 
Future work will attempt to extending this result by using the seminal results of Erd\"{o}s and Tur\'{a}n in \cite{ET50} concerning the discrepancy of roots of polynomials.

\section{Kazhdan's Property (T) Groups}
\noindent It is easy to give an equivalent definition of spectral envelopes of subsets of any discrete group. Future work will attempt to
generalize Feichtinger's conjecture for exponentials in this context and to exploit the exotic properties of groups, such as 
$SL(n,\Z), \, n \geq 3,$ that enjoy Kazhdan's Property (T) \cite{BHV08}, to construct a counterexample to this generalization thus showing that the answer to the Kadison-Singer problem is "no".

\end{document}